\documentclass[reqno]{amsart}
\usepackage{amscd}
\usepackage[dvips]{graphics}

\def\beq{\begin{equation}}
\def\eeq{\end{equation}}
\def\ba{\begin{array}}
\def\ea{\end{array}}

\numberwithin{equation}{section}
\newenvironment{abs}{\textbf{Abstract}\mbox{  }}{ }
\newenvironment{key words}{\textbf{Keywords}\mbox{  }}{ }
\newtheorem{theorem}{Theorem}[section]

\newtheorem{corollary}[theorem]{\textbf{Corollary}}
\newtheorem{proposition}[theorem]{\textbf{Proposition}}
\newtheorem{lemma}[theorem]{Lemma}
\renewenvironment{proof}{\noindent{\textbf{Proof.}}}{\hfill$\Box$}
\theoremstyle{remark}
\newtheorem{remark}[theorem]{\textbf{Remark}}
\theoremstyle{plain}
\begin{document}
\title{\textbf{Negative power nonlinear integral equations on bounded domains}}
\author[J. Dou, Q. Guo and M. Zhu]{{\it Dedicate to Ha\"im Brezis on occasion of his 75th birthday }
\\
\\Jingbo Dou, Qianqiao Guo and Meijun Zhu}
\address{Jingbo Dou, School of Mathematics and Information Science, Shaanxi Normal University, Xi'an, Shaanxi, 710119, China}

\email{jbdou@snnu.edu.cn}

\address{Qianqiao Guo, Department of Applied Mathematics, Northwestern Polytechnical University, Xi'an, Shaanxi, 710129, China}

\email{gqianqiao@nwpu.edu.cn}

\address{Meijun Zhu, Department of Mathematics,
The University of Oklahoma, Norman, OK 73019, USA}

\email{mzhu@math.ou.edu}


\maketitle

\noindent
\begin{abs}
This is the  continuation of our previous work \cite{DZ2017}, where we introduced and studied some  nonlinear integral equations on bounded domains that are related to the sharp Hardy-Littlewood-Sobolev  inequality. In this paper, we introduce some nonlinear integral equations on bounded domains that are related to the sharp reversed Hardy-Littlewood-Sobolev inequality. These are integral equations with nonlinear term involving negative exponents.  Existence results as well as nonexistence results are obtained.
\end{abs}\\
\begin{key words} Reversed sharp Hardy-Littlewood-Sobolev inequality, Integral equations, Existence, nonexistence
\end{key words}\\
\textbf{Mathematics Subject Classification(2000).}
45G10, 35J60 \indent
\section{\textbf{Introduction}\label{Section 1}}

In \cite{DZ2017}, motivated by the study of certain semi-linear equations and the sharp Sobolev inequality, we introduced and  studied the integral equations (with positive power) related to the sharp Hardy-Littlewood-Sobolev (HLS for short) inequality. Let us briefly recall these as the follows.


For $0<\alpha<n$,  on any bounded domain $\Omega \subset \mathbb{R}^n$ with smooth boundary, we considered
\begin{eqnarray*}
\hat{\xi}_\alpha(\Omega)=\sup_{f\in L^\frac{2n}{n+\alpha}(\Omega),f\neq 0}\frac{\int_{\Omega} \int_{\Omega} f(x)|x-y|^{-(n-\alpha)} f(y) dx dy}{||f||^2_{L^\frac{2n}{n+\alpha}(\Omega)}}.
\end{eqnarray*}
It was showed in \cite{DZ2017} that $
\hat{\xi}_\alpha (\Omega)= N_\alpha,$  where $N_\alpha$ is the best constant of the classical sharp HLS inequality (due to Lieb \cite{Lieb1983});
  And $\hat{\xi}_\alpha (\Omega) $ is not attained by any functions if $\Omega \ne \mathbb{R}^n$. This indicates that  there is not any energy maximizing  solution to
 $$
f^\frac{n-\alpha}{n+\alpha}(x)=\int_\Omega\frac{f(y)}{|x-y|^{n-\alpha}}dy,\quad f\ge 0, \quad x\in \overline \Omega.
$$
We then considered a general integral equation
\begin{equation}\label{HB}
f^{q-1}(x)=\int_\Omega\frac{f(y)}{|x-y|^{n-\alpha}}dy+\lambda \int_\Omega\frac {f(y)}{|x-y|^{n-\alpha-1}}dy,\quad f\ge 0, \quad x\in \overline \Omega,
\end{equation}
for $\alpha<n$, and studied  the existence and nonexistence of positive solutions for different power $q$ and parameter $\lambda$.

In this paper we consider integral equation \eqref{HB}  for $\alpha>n$.
This case is related to so called sharp reversed HLS inequality, which was discovered by Dou and Zhu \cite{DZ2015}.

%
%


Recall the sharp reversed  HLS inequality  from \cite{DZ2015} (see also related work by Ng\^{o}, Nguyen \cite{NN2017} and by Beckner \cite{Beck2015}):

\smallskip
\noindent  \textbf{Theorem A.} {\it  For $\alpha>n$,
\begin{equation}\label{R-HLS}
|\int_{\mathbb{R}^n} \int_{\mathbb{R}^n} f(x)|x-y|^{-(n-\alpha)} g(y) dx dy|\ge N_\alpha ||f||_{L^{\frac{2n}{n+\alpha}}(\mathbb{R}^n)}||g||_{L^{\frac {2n}{n+\alpha}}(\mathbb{R}^n)}
\end{equation}
holds for all non-negative functions $f, \ g \in L^{\frac{2n}{n+\alpha}}(\mathbb{R}^n),$
where
\begin{equation}\label{R-extreCon}
N_\alpha=N(\frac{2n}{n+\alpha},\alpha, n)=\pi^{(n-\alpha) /2} \frac{\Gamma(\alpha/2)}{\Gamma(n/2+\alpha/2)} \{\frac{\Gamma(n/2)}{\Gamma(n)} \}^{-\alpha/n};
\end{equation}
And the equality holds if and only if
\begin{equation}\label{R-HLS-ex}
f=c_1g=c_2\big(\frac{1}{c_3+|x-x_0|^2}\big)^{\frac {n+\alpha}2},
\end{equation}
where $c_1, c_2, c_3$ are any positive constants, and $x_0 \in \mathbb{R}^n$.

}\smallskip

\smallskip

Similar to what is done in \cite{DZ2017}: for any smooth domain $\Omega \subset \mathbb{R}^n$,  we consider
\begin{eqnarray*}
\xi_\alpha(\Omega)=\inf_{f\in L^\frac{2n}{n+\alpha}(\Omega),f\ge 0, f \ne 0}\frac{\int_{\Omega} \int_{\Omega} f(x)|x-y|^{-(n-\alpha)} f(y) dx dy}{||f||^2_{L^\frac{2n}{n+\alpha}(\Omega)}}.
\end{eqnarray*}
We will  show that
\begin{eqnarray}\label{R-energy}
 \xi_\alpha (\Omega)= N_\alpha,
  \end{eqnarray}
  and $\xi_\alpha (\Omega) $ is not attained by any functions if $\Omega \ne \mathbb{R}^n$ (see Proposition 2.1 below).
Again, we notice that the corresponding Euler-Lagrange equation for the minimizer (if the minimum is attained) is  integral equation \eqref{HB}  with $\lambda=0$ and a negative power ($\frac{n-\alpha}{n+\alpha}<0$).
We thus know that if $\alpha>n$ there is no energy minimizing solution to  integral equation \eqref{HB} for $q=2n/(n+\alpha)$ and $\lambda=0$.


 \medskip

 Let $p_\alpha=2n/(n-\alpha)$, $q_\alpha=2n/(n+\alpha)$ and $d(\Omega)=\sup\limits_{x,y\in\Omega}|x-y|$ be the diameter of the bounded domain $\Omega$.
In this paper, we consider  integral equation \eqref{HB} for $\alpha>n$. We shall prove
\begin{theorem}\label{main}
Assume $\alpha>n$ and $\Omega$ is a bounded domain with smooth boundary.

\noindent (1)  For $0<q<q_\alpha$(subcritical case), and  $-\frac{1}{d(\Omega)}<\lambda$, there is a positive solution $f\in C^1(\overline{\Omega})$ to equation \eqref{HB}.

\noindent(2) For $q=q_\alpha$ (critical case), and $-\frac{1}{d(\Omega)}< \lambda<0$, there is a positive solution $f\in C^1(\overline{\Omega})$ to equation \eqref{HB}.

\noindent(3) For $q_\alpha\le q<1$ (critical and supercritical cases), and   $\lambda\ge 0$, if  $\Omega$ is a star-shaped domain, then  there is not any positive $C^1(\overline{\Omega})$ solution to \eqref{HB}.

\end{theorem}


%
%

 We emphasis here that $q_\alpha<1$ (since $\alpha>n$). Thus equation \eqref{HB} has a nonlinear term with a negative power. Our results indicate that even though  the integral equation, which is related to the reversed HLS inequality, is of negative nonlinearity, similar phenomena to the integral equation  with positive power can be seen. Contrary to integral equations with positive nonlinearity, it seems that no compact embedding can be used directly for the existence result to the integral equation  with subcritical negative power. Different techniques are needed for deriving the existence as well as the nonexistence results. See more details in Section 3 below.

 We organize the paper as follows: In Section 2, we focus on the nonexistence result (part (3) of Theorem \ref{main}).  In Section 3, we first obtain the existence result  (part (1) of Theorem \ref{main}), and then we show the symmetric and monotonically increasing properties of solutions to the integral equation on a ball, even no pointwise boundary condition is given (Theorem \ref{UB-symmetry} below). In Section 4, we come back to the integral equations with critical exponent and a lower order term and prove the existence result (part (2) of Theorem \ref{main}).


 Notation: for any function $f(x)$ defined on $\Omega$, we always use $\tilde f(x)$ to represent its trivial extension in $\Bbb{R}^n$, namely,
\begin{eqnarray*}
 \tilde{f}(x):= \begin{cases}f(x)&\quad x\in\Omega,\\
 0&\quad x\in\mathbb{R}^n\backslash\Omega.
  \end{cases}
  \end{eqnarray*}
 We also denote
  $$
  I_\alpha f(x):=\int_{\Bbb{R}^n}\frac {f(y)}{|x-y|^{n-\alpha}} dy, \ \ \ \ I_{\alpha, \Omega} f(x):=\int_{\Omega}\frac {f(y)}{|x-y|^{n-\alpha}} dy
  $$
and
 $$  L_+^q(\Omega):=\{ f \in   L^q(\Omega)\setminus\{0\}  \ : \ f \ge 0\}.$$

%
%
%
%

\section{nonexistence for critical and supercritical cases}

In this section, we first derive energy estimate \eqref{R-energy} for any domain $\Omega \subset \Bbb{R}^n$ , and then show that the  infimum $\xi_\alpha(\Omega) $ is not achieved by any function once $\Omega \ne \Bbb{R}^n$.

\begin{proposition}\label{prop2-1} For any domain $\Omega \subset \Bbb{R}^n$,
$\xi_\alpha(\Omega)=N_\alpha$. Further, if  $\Omega \ne \Bbb{R}^n$, then the infimum $\xi_\alpha(\Omega)$  is not achieved by any function in $L^{q_\alpha}(\Omega)$.
\end{proposition}


\begin{proof}
If $f\in L_+^{q_\alpha}(\Omega)$, then $\tilde{f}\in L_+^{q_\alpha}(\mathbb{R}^n)$ . It follows that
\begin{eqnarray*}
\xi_\alpha(\Omega)&=&\inf_{{f}\in L_+^{q_\alpha}(\Omega)\setminus\{0\}}\frac{\int_{\mathbb{R}^n} \int_{\mathbb{R}^n} \tilde{f}(x)|x-y|^{-(n-\alpha)} \tilde{f}(y) dx dy}{||\tilde{f}||^2_{L^{q_\alpha}(\mathbb{R}^n)}}\\
&\ge & \inf_{{g}\in L_+^{q_\alpha}(\Bbb{R}^n)\setminus \{0\}}\frac{\int_{\mathbb{R}^n} \int_{\mathbb{R}^n} {g}(x)|x-y|^{-(n-\alpha)} {g}(y) dx dy}{||{g}||^2_{L^{q_\alpha}(\mathbb{R}^n)}} =N_\alpha.
\end{eqnarray*}

On the other hand,  recall  that $f(x)=\big(
\frac{1}{1+|x|^2}\big)^\frac{n+\alpha}2$ is an extremal function to the sharp reversed HLS inequality in {Theorem A}, as well as its conformal equivalent class:
\begin{equation}\label{fe}
f_\epsilon(x)=\epsilon^{-\frac{n+\alpha}2}f(\frac{|x-x_*|}\epsilon)=\big(\frac{\epsilon}{\epsilon^2+|x-x_*|^2}\big)^\frac{n+\alpha}2,
\end{equation}
where $x_*\in\mathbb{R}^n,\epsilon>0$. Thus
\[
\|I_\alpha f\|_{L^{p_\alpha}(\mathbb{R}^n)}=\|I_\alpha f_\epsilon\|_{L^{p_\alpha}(\mathbb{R}^n)},
\quad\|f\|_{L^{q_\alpha}(\mathbb{R}^n)}=\|f_\epsilon\|_{L^{q_\alpha}(\mathbb{R}^n)}.
\]
 Choose $x_0=x_*$ for some point $x_0\in \Omega$ and $R$ small enough so that $B_{R}(x_0)\subset\Omega$. Then we define test function $g (x)$ as
 \begin{eqnarray*}
g(x)= \begin{cases}f_\epsilon(x) &\quad x\in B_{R}(x_0)\subset\Omega,\\
 0&\quad x\in\mathbb{R}^n\backslash B_{R}(x_0).
  \end{cases}
  \end{eqnarray*}
Obviously, $g\in L_+^{q_\alpha}(\Bbb{R}^n).$ Thus,
\begin{eqnarray*}
& &\int_{\Omega} \int_{\Omega} \frac{g(x)g(y)}{|x-y|^{n-\alpha}} dx dy\\
&=&\int_{\mathbb{R}^n} \int_{\mathbb{R}^n} \frac{ f_\epsilon(x) f_\epsilon(y)}{|x-y|^{n-\alpha}} dx dy
-2\int_{\mathbb{R}^n} \int_{\mathbb{R}^n\backslash B_{R}(x_0) } \frac{ f_\epsilon(x) f_\epsilon(y)}{|x-y|^{n-\alpha}} dx dy\\
& &+\int_{\mathbb{R}^n\backslash B_{R}(x_0) } \int_{\mathbb{R}^n\backslash B_{R}(x_0) } \frac{ f_\epsilon(x) f_\epsilon(y)}{|x-y|^{n-\alpha}} dx dy\\
&\le&N_\alpha\|f_\epsilon\|^2_{L^{q_\alpha}(\mathbb{R}^n)}-2I_1 ,
\end{eqnarray*}
where
\begin{eqnarray*}
I_1&:=& \int_{\mathbb{R}^n} \int_{\mathbb{R}^n\backslash B_{R}(x_0) } \frac{ f_\epsilon(x) f_\epsilon(y)}{|x-y|^{n-\alpha}} dx dy.
\end{eqnarray*}
Notice that $f_\epsilon(x)$ is an extremal function for the sharp reversed HLS inequality. Thus it satisfies  integral equation:
\begin{equation*}
f^\frac{n-\alpha}{n+\alpha}(x)=B\int_{\mathbb{R}^n}\frac{f(y)}{|x-y|^{n-\alpha}}dy,
\end{equation*} where $B$ 
is a suitable positive constant.
We thus can estimate $I_1$ in the following:
\begin{eqnarray*}
I_1=
C\int_{\mathbb{R}^n\backslash B_{R}(x_0) } f_\epsilon^{\frac{2n}{n+\alpha}}(x) dx
=O(\frac{R}\epsilon)^{-n},  \quad\quad\text{as}\quad\epsilon\to0.
\end{eqnarray*}
And we also have
\begin{eqnarray*}
\int_{B_{R}(x_0) } f_\epsilon^{\frac{2n}{n+\alpha}}(x) dx&=&\int_{\mathbb{R}^n} f_\epsilon^{\frac{2n}{n+\alpha}}(x) dx-\int_{\mathbb{R}^n\backslash B_{R}(x_0) } f_\epsilon^{\frac{2n}{n+\alpha}}(x) dx\\
&=&\int_{\mathbb{R}^n} f_\epsilon^{\frac{2n}{n+\alpha}}(x) dx-O(\frac{R}\epsilon)^{-n},  \quad\quad\text{as}\quad\epsilon\to0.
\end{eqnarray*}
Hence, for small enough $\epsilon>0$, we have
 \begin{eqnarray*}
 \xi_\alpha(\Omega)&\le &\frac{\int_{\Omega} \int_{\Omega} \frac{g(x)g(y)}{|x-y|^{n-\alpha}} dx dy}{\|g\|^2_{L^{q_\alpha}(\Omega)}}\\
 &\le &\frac{N_\alpha\|f_\epsilon\|^2_{L^{q_\alpha}(\mathbb{R}^n)}-I_1}{\|f_\epsilon\|^2_{L^{q_\alpha}(B_{R}(x_0))}}\\
 &=&\frac{N_\alpha\|f_\epsilon\|^2_{L^{q_\alpha}(\mathbb{R}^n)}-O(\frac{R}\epsilon)^{-n}}{\|f_\epsilon\|^2_{L^{q_\alpha}(\mathbb{R}^n)}-O(\frac{R}\epsilon)^{-n}} ,
  \end{eqnarray*}
which yields $\xi_\alpha(\Omega)\le N_\alpha$ as $\epsilon \to 0$.

Finally,  we show  that $\xi_\alpha(\Omega)$ is not achieved if $\Omega\neq\mathbb{R}^n$.
In fact, if $\xi_\alpha(\Omega)$ were attained by some function $u\in L_+^{q_\alpha}(\Omega)$, then $\tilde u\in L_+^{q_\alpha}(\Bbb{R}^n)$ would be an extremal function to the sharp reversed HLS inequality on $\Bbb{R}^n$, which is impossible due to Theorem A.
%
\end{proof}

\medskip

Proposition \ref{prop2-1} indicates that for $\alpha>n$ there is not any minimizing  energy solution to \eqref{HB} for $q=2n/(n+\alpha)$ and $\lambda=0$. In fact, we will show that there is not any positive $C^1$ solution to \eqref{HB}  for $\alpha>n$, $q=2n/(n+\alpha)$ and $\lambda=0$ on any star-shaped domain.

\medskip

\noindent{\bf Proof of part (3) in Theorem \ref{main} (nonexistence part)}. Without loss of generality, here we assume that the origin is in $\Omega$ and the domain is star-shaped with respect to the origin.

Recall the following Pohozaev identity from  \cite{DZ2017}.

\begin{lemma}\label{poh}  Assume that the origin is in $\Omega$ and the domain is star-shaped with respect to the origin. If $u\in C^1(\overline \Omega)$ is a non-negative solution to
\begin{equation}\label{gen_equ}
u(x)=\int_\Omega\frac{u^{p-1}(y)}{|x-y|^{n-\alpha}}dy+ \lambda \int_\Omega\frac{u^{p-1}(y)}{|x-y|^{n-\alpha-1}}dy,\quad x\in \overline\Omega,
\end{equation}
where $p\ne 0$, $\lambda \in \Bbb{R}$,
then
\begin{eqnarray}\label{SB-5}
(\frac{n}{p}+\frac{\alpha-n}2)\int_{\Omega}u^{p}(x)dx= -\frac{\lambda}2\int_{\Omega}\int_{\Omega}\frac{u^{p-1}(x)u^{p-1}(y)}{|x-y|^{n-\alpha}}dydx+\frac1{p}\int_{\partial \Omega}(x\cdot \nu) u^{p}(x)d\sigma,
\end{eqnarray}
where $\nu$ is the outward unit normal vector to $\partial \Omega$.
\end{lemma}

\medskip

Applying Lemma \ref{poh} to \eqref{HB} for $\lambda\ge0$ and $u(x)=f^{q-1}(x)$. Thus $p=\frac{q}{q-1}=q^\prime$. Noticing  that $1>q\ge q_\alpha$ is equivalent to  $p\le p_\alpha<0$, we have
\begin{eqnarray*}
-\frac{\lambda}2\int_{\Omega}\int_{\Omega}\frac{u^{p-1}(x)u^{p-1}(y)}{|x-y|^{n-\alpha-1}}dydx+\frac 1p \int_{\partial \Omega}(x\cdot\nu) u^{p-1}(x)d\sigma\ge 0.
\end{eqnarray*}
Since $\Omega$ is star-shaped domain about the origin, we have $x\cdot \nu>0$  on $\partial\Omega$.   If $\lambda>0$, then $u(x)\equiv\infty $ on $\overline{\Omega}$. If $\lambda=0$, then $u\equiv\infty$ on $\partial\Omega$. Therefore we obtain a contradiction to that $f(x)$ is a positive $C^1(\overline{\Omega})$ solution.

 \medskip

 \begin{remark} Condition $q<1$ is needed in our proof.
 If $q>1$, one  may call it a supercritical exponent  (since  it is bigger than $q_\alpha$). However, in this case $p=q/(q-1)>0$ is also bigger than negative $p_\alpha$. Our nonexistence result may not be true any more in this case.
 \end{remark}


\smallskip

Note that the unit ball is conformally equivalent to the upper half space. We have

\begin{corollary}\label{Pro-Half space}
There is no positive solution $u\in L^{p_\alpha}(\overline{\mathbb{R}^n_+}) \cap C^1(\overline {\mathbb{R}^n_+}) $ to
 $$
u(x)=\int_{\mathbb{R}^n_+}\frac{u^{p_\alpha-1}(y)}{|x-y|^{n-\alpha}}dy,\quad x\in \overline{\mathbb{R}^n_+}.
$$
\end{corollary}


\section{Existence result for subcritical case}

For subcritical exponents we have the following inequality:
\begin{lemma}\label{bound HLS}
Let $q\in (0, {q_\alpha})$.
 There exists a positive constant  $C(n,q,\alpha,\Omega)>0$ such that
 \begin{equation}\label{B-HLS inequality}
 \int_{\Omega} \int_{\Omega} {f}(x)|x-y|^{-(n-\alpha)} {f}(y) dx dy
\ge C(n,q,\alpha,\Omega)\|f\|^2_{L^q(\Omega)}
 \end{equation} holds for any non-negative function  $f\in L^q(\Omega)$.
 \end{lemma}
 \begin{proof}
For $f\in L^q(\Omega)$, by using the reversed HLS inequality \eqref{R-HLS} we have
\begin{eqnarray*}
\langle I_{\alpha, \Omega} f,f\rangle&=& \int_{\mathbb{R}^n} \int_{\mathbb{R}^n} \tilde{f}(x)|x-y|^{-(n-\alpha)} \tilde{f}(y) dx dy\\
&\ge& N_\alpha \|\tilde{f}\|^2_{L^{q_\alpha}(\mathbb{R}^n)}=N_\alpha \|\tilde{f}\|^2_{L^{q_\alpha}(\Omega)}\\
&\ge & C(n, q, \alpha, \Omega)\|f\|^2_{L^q(\Omega)}.
\end{eqnarray*}
\end{proof}

We would like to point out that one can also prove the above lemma directly via a Young type inequality as  that in Dou, Guo and Zhu \cite{DGZ17}.







Based on the above lemma, we can obtain the existence result for subcritical exponent (part (1) of Theorem \ref{main}). 
 Notice that it is different from the case $0<\alpha<n$ (Lemma 3.2 in \cite{DZ2017}) since no compact embedding can be used directly here. We follow a similar approach used in Dou, Guo and Zhu \cite{DGZ17}.

\begin{lemma}\label{Bound-attained} For $0<q<{q_\alpha}, \lambda>-\frac{1}{d(\Omega)}$, infimum
\begin{equation*}
\xi_{\alpha,q}(\Omega):=\inf_{f\in L_+^q(\Omega)}\frac{\int_{\Omega} \int_{\Omega}(f(x)|x-y|^{-(n-\alpha)} f(y)+\lambda f(x)|x-y|^{-(n-\alpha-1)} f(y)) dy dx}{\|f\|^2_{L^q(\Omega)}}>0,
\end{equation*}
 and it is attained by some nonnegative  function in $L_+^q(\Omega).$
\end{lemma}

\begin{proof}  Notice:  for $x, \ y \in \Omega$ and $0>\lambda>-1/d(\Omega),$ $1+\lambda|x-y|\ge 1+\lambda d(\Omega)>0$.  Thus we know $\xi_{\alpha,q}(\Omega)>0$ by Lemma \ref{bound HLS}.

 Choose a minimizing nonnegative sequence $\{f_j\}_{j=1}^\infty$ in $L^q(\Omega)$. Assume without loss of generality that $f_j\in L^{q_\alpha}(\Omega)$ (see, for example, Proposition 2.5 in \cite{DGZ17}). Then we can normalize it such that $\|f_j\|_{L^{q_\alpha}(\Omega)}=1$. It follows that  there exists a subsequence such that
\[f_j^q\rightharpoonup f_*^q\quad\text{weakly\, in}\quad L^{\frac{q_\alpha}{q}}(\Omega), \ \ \mbox{as} \ j\to\infty.
\]
Then
\begin{equation}\label{equ1-Bound-attained}
\int_{\Omega}f_j^q\to \int_{\Omega}f_*^q, \ \ \mbox{as} \ j\to\infty.
\end{equation}

Claim: $\|f_j\|_{L^{1}(\Omega)}\le C.$

We relegate the proof of this claim to the end.

Once the claim is proved, we have $\int_{\Omega}f_*^q>C>0$ via an interpolation inequality
and $f_j^q\rightharpoonup f_*^q\quad\text{weakly\, in}\quad L^{\frac{1}{q}}(\Omega).$
Then for any fixed $x\in\overline{\Omega}$, $f_*^{1-q}(y)|x-y|^{\alpha-n}(1+\lambda|x-y|)\in L^{\frac{1}{1-q}}(\Omega)$, thus, as $j \to \infty$,
\begin{eqnarray*}
\int_{\Omega}f_j^q(y)f_*^{1-q}(y)|x-y|^{\alpha-n}(1+\lambda|x-y|)dy\to\int_{\Omega} f_*(y)|x-y|^{\alpha-n}(1+\lambda|x-y|)dy.
\end{eqnarray*}

Further, we show  that the above convergence is actually uniformly convergent for  all $x\in\overline\Omega$.

By H\"{o}lder's inequality we have
\begin{eqnarray*}
&&\int_{\Omega}f_j^q(y)f_*^{1-q}(y)|x-y|^{\alpha-n}(1+\lambda|x-y|)dy\\
&\le & (\int_{\Omega}f_j^{q\cdot\frac{1}{q}}(y)dy)^q(\int_{\Omega}(f_*^{1-q}(y)|x-y|^{\alpha-n}(1+\lambda|x-y|))^\frac{1}{1-q}dy)^{1-q}\le C,
\end{eqnarray*}
that is, $\int_{\Omega}f_j^q(y)f_*^{1-q}(y)|x-y|^{\alpha-n}(1+\lambda|x-y|)dy$ is uniformly bounded for $x\in\overline{\Omega}$. Notice that for any $x_1,x_2,y\in\overline{\Omega},$
\begin{equation*}
||x_1-y|^{\alpha-n}-|x_2-y|^{\alpha-n}|
\le \begin{cases}
C|x_1-x_2|^{\alpha-n},\ \mbox{if}~0<\alpha-n\le 1,\\
C|x_1-x_2|,\ \ \ \ \ \ \mbox{if}~\alpha-n>1.
\end{cases}
\end{equation*}
Then for any $x_1,x_2\in\overline{\Omega}, $
\begin{eqnarray*}
&&\ \ |\int_{\Omega}f_j^q(y)f_*^{1-q}(y)|x_1-y|^{\alpha-n}(1+\lambda|x_1-y|)dy\\
&&-\int_{\Omega}f_j^q(y)f_*^{1-q}(y)|x_2-y|^{\alpha-n}(1+\lambda|x_2-y|)dy|\\
&\le &\int_{\Omega}f_j^q(y)f_*^{1-q}(y)||x_1-y|^{\alpha-n}(1+\lambda|x_1-y|)-|x_2-y|^{\alpha-n}(1+\lambda|x_2-y|)|dy
\\
&\le &C\max(|x_1-x_2|^{\alpha-n}, |x_1-x_2|)\int_{\Omega}f_j^q(y)f_*^{1-q}(y)dy
\\
&\le &C\max(|x_1-x_2|^{\alpha-n}, |x_1-x_2|)(\int_{\Omega}f_j(y)dy)^q(\int_{\Omega}f_*(y)dy)^{1-q}\\
&\le& C\max(|x_1-x_2|^{\alpha-n}, |x_1-x_2|).
\end{eqnarray*}
Thus $\int_{\Omega}f_j^q(y)f_*^{1-q}(y)|x-y|^{\alpha-n}(1+\lambda|x-y|)dy$ is equicontinuous in $ \overline{\Omega}$. It follows that, as $j\to\infty$,
\[\int_{\Omega}f_j^q(y)f_*^{1-q}(y)|x-y|^{\alpha-n}(1+\lambda|x-y|)dy\to\int_{\Omega} f_*(y)|x-y|^{\alpha-n}(1+\lambda|x-y|)dy
\]
uniformly for $x\in\overline{\Omega}$.

Therefore for any $\epsilon>0$ small enough, there exists $j_0\in \mathbb{N}$ such that for any $j>j_0$,
\begin{eqnarray*}
&&\int_{\Omega}\int_{\Omega}f_j^q(x)f_*^{1-q}(x)|x-y|^{\alpha-n}(1+\lambda|x-y|)f_j^q(y)f_*^{1-q}(y)dxdy\\
&\ge& \int_{\Omega}f_j^q(x)f_*^{1-q}(x)[\int_{\Omega}|x-y|^{\alpha-n}(1+\lambda|x-y|)f_*(y)dy-\epsilon]dx.
\end{eqnarray*}
By H\"{o}lder's inequality we know $\int_{\Omega}f_j^q(x)f_*^{1-q}(x)dx\le C$.  Again,  notice that
\begin{eqnarray*}
&&\int_{\Omega}f_j^q(x)f_*^{1-q}(x)\int_{\Omega}|x-y|^{\alpha-n}(1+\lambda|x-y|)f_*(y)dydx \\
&\to& \int_{\Omega}\int_{\Omega}f_*(x)|x-y|^{\alpha-n}(1+\lambda|x-y|)f_*(y)dydx
\end{eqnarray*}
since $f_*^{1-q}(x)\int_{\Omega}|x-y|^{\alpha-n}(1+\lambda|x-y|)f_*(y)dy\in L^{\frac{1}{1-q}}(\Omega)$. We have for $j>j_0$ large enough,
\begin{eqnarray*}
&&\int_{\Omega}\int_{\Omega}f_j^q(x)f_*^{1-q}(x)|x-y|^{\alpha-n}(1+\lambda|x-y|)f_j^q(y)f_*^{1-q}(y)dxdy\\
&\ge& \int_{\Omega}\int_{\Omega}f_*(x)|x-y|^{\alpha-n}(1+\lambda|x-y|)f_*(y)dydx-\epsilon-C\epsilon.
\end{eqnarray*}
Similarly, we also have, for $j>j_0$ large enough,
\begin{eqnarray*}
&&\int_{\Omega}\int_{\Omega}f_j^q(x)f_*^{1-q}(x)|x-y|^{\alpha-n}(1+\lambda|x-y|)f_j^q(y)f_*^{1-q}(y)dxdy\\
&\le& \int_{\Omega}\int_{\Omega}f_*(x)|x-y|^{\alpha-n}(1+\lambda|x-y|)f_*(y)dydx+\epsilon+C\epsilon.
\end{eqnarray*}
Therefore, as $j \to \infty$,
\begin{eqnarray*}
&&\int_{\Omega}\int_{\Omega}f_j^q(x)f_*^{1-q}(x)|x-y|^{\alpha-n}(1+\lambda|x-y|)f_j^q(y)f_*^{1-q}(y)dxdy\\
&\to& \int_{\Omega}\int_{\Omega}f_*(x)|x-y|^{\alpha-n}(1+\lambda|x-y|)f_*(y)dydx.
\end{eqnarray*}
It follows from the above and H\"{o}lder's inequality that
\begin{eqnarray*}
&&\liminf\limits_{j\to\infty}\int_{\Omega}\int_{\Omega}f_j(x)|x-y|^{\alpha-n}(1+\lambda|x-y|)f_j(y)dxdy\\
&\ge& \int_{\Omega}\int_{\Omega} f_*(x)|x-y|^{\alpha-n}(1+\lambda|x-y|)f_*(y)dxdy.
\end{eqnarray*}
Since $||f_j||_{L^q(\Omega)} \to ||f_*||_{L^q(\Omega)} >0$, the above inequality then implies
\begin{eqnarray*}
&&\liminf\limits_{j\to\infty}\frac{\int_{\Omega}\int_{\Omega}f_j(x)|x-y|^{\alpha-n}(1+\lambda|x-y|)f_j(y)dxdy}{\|f_j\|^2_{L^q(\Omega)}}\\
&\ge&\frac{\int_{\Omega}\int_{\Omega} f_*(x)|x-y|^{\alpha-n}(1+\lambda|x-y|)f_*(y)dxdy}{\|f_*\|^2_{L^q(\Omega)}}.
\end{eqnarray*}
That is: $f_*$ is a minimizer.

Now we are left to prove the claim: $\|f_j\|_{L^{1}(\Omega)}\le C.$


From  $\|f_j\|_{L^{q_\alpha}(\Omega)}=1$, we have
\begin{eqnarray*}
&&\int_{\Omega}\int_{\Omega}f_j(x)|x-y|^{\alpha-n}(1+\lambda|x-y|)f_j(y)dxdy\\
&\ge&(1-|\lambda| d(\Omega))\int_{\Omega}\int_{\Omega}f_j(x)|x-y|^{\alpha-n}f_j(y)dxdy\\
&\ge&  (1-|\lambda| d(\Omega)) N_\alpha\|f_j\|_{L^{q_\alpha}(\Omega)}^2\\
&\ge& C>0.
\end{eqnarray*}
 Since  $\{f_j\}_{j=1}^\infty$ is a minimizing nonnegative sequence, we conclude that  $\|f_j\|_{L^{q}(\Omega)}\ge C_1(n,\alpha,q)>0$.
By H\"{o}lder's inequality, we have
\begin{eqnarray*}
0<C_1(n,\alpha,q)\le\|f_j\|_{L^{q}(\Omega)}\le C_2(n,\alpha,q)\|f_j\|_{L^{q_\alpha}(\Omega)}=C_2(n,\alpha,q).
\end{eqnarray*}
The upper bound on $ ||f_j\|_{L^{q}(\Omega)} $ indicates that
$$
\int_{\Omega}\int_{\Omega}f_j(x)|x-y|^{\alpha-n}f_j(y)dxdy<C_3(n,\alpha,q).
$$
It follows, via reversed H\"{o}lder's inequality, that
 $$\|I_{\alpha, \Omega} f_j\|_{L^{q^\prime}(\Omega)}\le C_4(n,\alpha,q)<\infty.$$

Further,  H\"{o}lder's inequality and reversed HLS inequality yield  that
\begin{eqnarray*}
\infty>C_4(n,\alpha,q) &\ge &\|I_{\alpha, \Omega} f_j\|_{L^{q^\prime}(\Omega)}\\
&\ge & C_5(n,\alpha,q) \|I_{\alpha, \Omega} f_j\|_{L^{p_\alpha}(\Omega)}\\
&\ge & C_6(n,\alpha,q) \|f_j\|_{L^{q_\alpha}(\Omega)}= C_6(n,\alpha,q)>0.
\end{eqnarray*}

Then we can show that for $M>0$ such that $M^{q^\prime}|\Omega|<\frac{1}{2} (C_4(n,\alpha,q))^{q^\prime}$, there exists $0<\delta<|\Omega|$ such that
\begin{equation}
\label{add-1}
m\{x: I_{\alpha, \Omega} f_j(x)\le M\}>\delta, ~{for~all} \ \ j.
\end{equation}
In fact, for  $\Omega_1:=\{x: I_{\alpha, \Omega} f_j(x)\le M\}$,
\begin{eqnarray*}
&&(C_4(n,\alpha,q))^{q^\prime}\\
&\le&
\int_{\Omega} (I_{\alpha, \Omega} f_j)^{q^\prime}dx=\int_{\Omega\setminus\Omega_1} (I_{\alpha, \Omega} f_j)^{q^\prime}dx+\int_{\Omega_1} (I_{\alpha, \Omega} f_j)^{q^\prime}dx\\
&\le & M^{q^\prime}|\Omega|+(\int_{\Omega_1} (I_{\alpha, \Omega} f_i)^{p_\alpha}dx)^{\frac{q^\prime}{p_\alpha}}|\Omega_1|^{1-\frac{q^\prime}{p_\alpha}}\\
&\le & M^{q^\prime}|\Omega|+(\int_{\Omega} (I_{\alpha, \Omega} f_i)^{p_\alpha}dx)^{\frac{q^\prime}{p_\alpha}}|\Omega_1|^{1-\frac{q^\prime}{p_\alpha}}\\
&\le & M^{q^\prime}|\Omega|+(\frac{C_6(n,\alpha,q)}{C_5(n,\alpha,q)})^{q^\prime}|\Omega_1|^{1-\frac{q^\prime}{p_\alpha}}\\
&<&\frac{1}{2}(C_4(n,\alpha,q))^{q^\prime}+|\Omega_1|^{1-\frac{q^\prime}{p_\alpha}}(\frac{C_6(n,\alpha,q)}{C_5(n,\alpha,q)})^{q^\prime},
\end{eqnarray*}
which yields the existence of such $\delta>0$.

Due to \eqref{add-1} we know that  there exists $\epsilon_0>0$,  such that for any $j$, we can find two points $x_j^1, x_j^2\in \Omega$ with the properties that $|x_j^1-x_j^2|\ge\epsilon_0$ and $$I_{\alpha, \Omega} f_j(x_j^1)=\int_{\Omega} f_j(y)|x_j^1-y|^{\alpha-n}\le M,~ I_{\alpha, \Omega} f_j(x_j^2)=\int_{\Omega} f_j(y)|x_j^2-y|^{\alpha-n}\le M.$$
So
\begin{eqnarray*}
\int_{\Omega}f_j(y)dy&\le & \int_{\Omega \setminus B(x_j^1, \frac{\epsilon_0}{4})}f_j(y)dy+\int_{\Omega \setminus B(x_j^2, \frac{\epsilon_0}{4})}f_j(y)dy\\
&\le& (\frac{4}{\epsilon_0})^{\alpha-n}\int_{\Omega \setminus B(x_j^1, \frac{\epsilon_0}{4})}f_j(y)|x_j^1-y|^{\alpha-n}dy\\
&&+(\frac{4}{\epsilon_0})^{\alpha-n}\int_{\Omega \setminus B(x_j^1, \frac{\epsilon_0}{4})}f_j(y)|x_j^2-y|^{\alpha-n}dy\\
&\le&(\frac{4}{\epsilon_0})^{\alpha-n}2M,
\end{eqnarray*}
uniformly for all $j$. We thus verify the claim, and hereby, complete the proof of Lemma \ref{Bound-attained}.
\end{proof}

\smallskip

It is standard to check that  the  minimizer $f(x)$ for energy $\xi_{\alpha,q}(\Omega)$ is positive, and, up to a constant multiplier, satisfies the following equation:
\begin{equation}\label{EL-equ-2}
f^{q-1}(x)=\int_\Omega\frac{f(y)}{|x-y|^{n-\alpha}}dy+\lambda\int_\Omega\frac{f(y)}{|x-y|^{n-\alpha-1}}dy,\quad x\in\overline\Omega.
\end{equation}
From the proof of Lemma \ref{Bound-attained}, we also know that $f(x) \in L^1(\Omega)$, thus $f^{q-1} (x) \in L^\infty(\Omega)$ by equation \eqref{EL-equ-2}.

 Writing $u(x)=f^{q-1}(x), p=q'$,  we thus find a weak positive solution $u(x)\in L^p(\Omega)$ to
\begin{equation}\label{EL-equ-3}
u(x)=\int_\Omega\frac{u^{p-1}(y)}{|x-y|^{n-\alpha}}dy+\lambda\int_\Omega\frac{u^{p-1}(y)}{|x-y|^{n-\alpha-1}}dy,\quad x\in\overline\Omega
\end{equation}
for $0>p>\frac{2n}{n-\alpha}=p_\alpha.$ To complete the proof of part (1) in Theorem \ref{main}, we need to show that $u \in C^1(\overline \Omega)$.
%

In fact, $f(x)\in L^1(\Omega)$ implies
$$\int_\Omega u^{p-1}(y)dy<\infty.$$
It is easy to see that $u \in C(\overline \Omega)$ from equation \eqref{EL-equ-3}. To show $u\in C^1(\overline{\Omega})$, we can  directly compute, for $i=1,..., n$,
$$\partial_{x_i}u(x)=-(n-\alpha)\int_{\Omega}\frac{u^{p-1}(y)(x_i-y_i)}{|x-y|^{n-\alpha+2}}dy \in C(\overline{\Omega}).
$$
Part (1) of Theorem \ref{main} is hereby proved.


\medskip


  It is interesting to study some properties about the  positive solutions to the new integral equation \eqref{EL-equ-3}(such as multiplicity of solutions, blowup behavior as $q \to q_\alpha$ in a star-shaped domain, etc.)
  In the rest of this section,  as in \cite{DZ2017}, we will  show that
even though the boundary condition is not given pointwise,  the symmetric property for solutions to the integral equation \eqref{EL-equ-3} with $\lambda=0$  on a unit ball still holds. Contrary to the result in \cite{DZ2017}, here we will show that the solution is {\it monotone increasing} due to the monotone increasing property of the kernel.


On $B_1:=B_1(0)=\{x\in\mathbb{R}^n\,|\, |x|<1,x\in\mathbb{R}^n\},$ for  $\lambda=0$ we rewrite the equation \eqref{EL-equ-3} as
\begin{equation}\label{UB-2}
u(x)=\int_{B_1}\frac{u^{p-1}(y)}{|x-y|^{n-\alpha}}dy,\quad x\in B_1(0).
\end{equation}
 We have
 \begin{theorem}\label{UB-symmetry}
Let $\alpha>n$, $p\in(p_\alpha, 0)$. Then every positive solution $u\in L^p (\overline{B_1})$ to \eqref{UB-2} is radially symmetric about the origin and strictly increasing in the radial direction.
\end{theorem}

Easy to see from the proof of Part (1) of Theorem \ref{main}  that $u\in C^{1}(\overline{B_1})$. We will use the method of moving planes to  prove Theorem \ref{UB-symmetry}.


Firstly, we recall the idea of the method of moving planes in $B_1$(see e.g. \cite{CFY2014,DZ2017}).

For any real number $\lambda\in (-1,0)$, define $T_\lambda=\{x\in \mathbb{R}^n \ | \ x_1=\lambda\},$
and  $x^\lambda=(2\lambda-x_1,x_2,\cdots,x_n)$ as the reflection of point $x=(x_1,x_2,\cdots,x_n)$ about  plane $T_\lambda$.
Let
\begin{eqnarray*}
\Sigma_\lambda=\{x=(x_1,x_2,\cdots,x_n)\in B_1|-1<x_1<\lambda\},
\end{eqnarray*}
and $\Sigma^C_\lambda=B_1\backslash \overline \Sigma_\lambda$ be the complement of $\Sigma_\lambda$ in $B_1$. Set $u_\lambda(x)=u(x^\lambda).$
We shall  complete the proof  in two steps. In step 1, we show that for $\lambda$ sufficiently close to $-1$,
\begin{equation}\label{UB-6}
u(x)\ge u_\lambda(x),  \quad \forall x\in \Sigma_\lambda.
\end{equation}
Then we can start to move  plane $T_\lambda$  along the $x_1$ direction. In step 2, we move the plane to the right as long as inequality \eqref{UB-6} holds. We show that the plane can be  moved to $\lambda=0$. So
\begin{equation}\label{UB-7}
u(-x_1,x_2,\cdots,x_n)\geq u(x_1,x_2,\cdots,x_n), \quad\forall x\in B_1,x_1\geq 0.
\end{equation}
Similarly, we can start to move  plane $T_\lambda$ from a place close to $\lambda=1$, and move it to the left limiting position $T_0$. Then
\begin{equation}\label{UB-8}
u(-x_1,x_2,\cdots,x_n)\leq u(x_1,x_2,\cdots,x_n),  \quad\forall x\in B_1,x_1\geq 0.
\end{equation}
By \eqref{UB-7} and \eqref{UB-8}, we have that $u(x)$ is symmetric about the plane $x_1=0$. Similarly, we can show that $u(x)$ is symmetric about any plane passing through the origin,   which then implies that $u(x)$ is radially symmetric about the origin and strictly increasing in the radial direction.

First, we have following comparison inequality.
\begin{lemma}\label{Lm-symmetry-1}
For any $x\in \Sigma_\lambda$ with $\lambda\in (-1,0)$, it holds
\begin{equation}\label{UB-3}
u(x)-u_\lambda(x)\geq\int_{\Sigma_\lambda}\big[\frac1{|x-y|^{n-\alpha}}-\frac1{|x^\lambda-y|^{n-\alpha}}\big]({u^{p-1}(y)}-{u^{p-1}_\lambda(y)})dy.
\end{equation}
\end{lemma}

\begin{proof}
Let $\widetilde{\Sigma}_\lambda=\{x^\lambda~|~x\in{\Sigma}_\lambda\}$ be the reflection of $\Sigma_\lambda$ about  plane $T_\lambda$, then
\begin{eqnarray*}
u(x)&=&\int_{\Sigma_\lambda}\frac{u^{p-1}(y)}{|x-y|^{n-\alpha}}dy+\int_{\widetilde{\Sigma}_\lambda}\frac{u^{p-1}(y)}{|x-y|^{n-\alpha}}dy\\
&+&\int_{\Sigma^C_\lambda\backslash\widetilde{\Sigma}_\lambda}\frac{u^{p-1}(y)}{|x-y|^{n-\alpha}}dy.
\end{eqnarray*}
Noting that $|x-y|>|x^\lambda-y|$ in $x\in \Sigma_\lambda, y\in\Sigma^C_\lambda\backslash\widetilde{\Sigma}_\lambda $, we have
\begin{eqnarray*}
u(x)-u_\lambda(x)&=&\int_{\Sigma_\lambda}\big[\frac1{|x-y|^{n-\alpha}}-\frac1{|x^\lambda-y|^{n-\alpha}}\big]{u^{p-1}(y)}dy\\
&&+\int_{\widetilde{\Sigma}_\lambda}\big[\frac1{|x-y|^{n-\alpha}}-\frac1{|x^\lambda-y|^{n-\alpha}}\big]{u^{p-1}(y)}dy\\
&&+\int_{\Sigma^C_\lambda\backslash\widetilde{\Sigma}_\lambda}\big[\frac1{|x-y|^{n-\alpha}}-\frac1{|x^\lambda-y|^{n-\alpha}}\big]{u^{p-1}(y)}dy\\
&\geq&\int_{\Sigma_\lambda}\big[\frac1{|x-y|^{n-\alpha}}-\frac1{|x^\lambda-y|^{n-\alpha}}\big]{u^{p-1}(y)}dy\\
&&+\int_{\Sigma_\lambda}\big[\frac1{|x-y^\lambda|^{n-\alpha}}-\frac1{|x^\lambda-y^\lambda|^{n-\alpha}}\big]{u^{p-1}_\lambda(y)}dy\\
&=&\int_{\Sigma_\lambda}\big[\frac1{|x-y|^{n-\alpha}}-\frac1{|x^\lambda-y|^{n-\alpha}}\big]{u^{p-1}(y)}dy\\
&&-\int_{\Sigma_\lambda}\big[\frac1{|x-y|^{n-\alpha}}-\frac1{|x^\lambda-y|^{n-\alpha}}\big]{u^{p-1}_\lambda(y)}dy\\
&=&\int_{\Sigma_\lambda}\big[\frac1{|x-y|^{n-\alpha}}-\frac1{|x^\lambda-y|^{n-\alpha}}\big]({u^{p-1}(y)}-{u_\lambda^{p-1}(y)})dy.
\end{eqnarray*}
\end{proof}


\noindent \textbf{Proof of Theorem \ref{UB-symmetry}.}

\noindent  \textbf{Step 1.}  Let $u\in C^1(\overline{B_1})$ be a positive solution to equation  \eqref{UB-2}.  We show that for $\lambda$  sufficiently close to $-1$, inequality \eqref{UB-6} holds.

From \eqref{UB-2} we have
\begin{eqnarray*}
\frac{\partial u(x)}{\partial x_1}|_{x_1=-1}
&=&(\alpha-n)\int_{B_1}|x-y|^{\alpha-n-2}(-1-y_1)u^{p-1}(y)dy\\
&<&0.
\end{eqnarray*}
Therefore for $\lambda$ sufficiently close to $-1$, we have
\begin{eqnarray*}
 u(x) \ge u_\lambda(x)     \quad\text{for} ~ x\in\Sigma_\lambda.
\end{eqnarray*}

\noindent \textbf{Step 2.} Plane $T_\lambda$ can be moved continuously towards right to its limiting position as long as inequality \eqref{UB-6} holds.

Define
\[\lambda_0=\sup\{\lambda\in [-1,0)\ | \ u(y)\geq u_\mu(y),\forall y \in\Sigma_\mu, \ \mu\leq\lambda\}.\]
We claim that $\lambda_0$ must be 0.

We prove it by contradiction. Suppose not, that is,  $\lambda_0<0$.

 We first  show that
\[u(x)>u_{\lambda_0}(x),\quad\text{in} ~\Sigma_{\lambda_0}.\]
Hence, we have
\[u(x)-u_{\lambda_0}(x)>c_1>0,\quad \text{in} ~\Sigma_{\lambda_0-\epsilon_1}
\]for $\epsilon_1>0$ small enough.

In fact, since $|x-y|<|x-y^{\lambda_0}|$ for $x, y\in{\Sigma}_{\lambda_0}$, we have, similar to the calculation in the proof of  Lemma \ref{Lm-symmetry-1}, that
\begin{eqnarray}
u(x)-u_{\lambda_0}(x)
&=&\int_{\Sigma_{\lambda_0}}\big[\frac1{|x-y|^{n-\alpha}}-\frac1{|x^{\lambda_0}-y|^{n-\alpha}}\big]({u^{p-1}(y)}-{u_{\lambda_0}^{p-1}(y)})dy\nonumber\\
&+&\int_{\Sigma^C_{\lambda_0}\backslash\widetilde{\Sigma}_{\lambda_0}}\big[\frac1{|x-y|^{n-\alpha}}-\frac1{|x^{\lambda_0}-y|^{n-\alpha}}\big]{u^{p-1}(y)}dy\nonumber\\
&\ge&\int_{\Sigma^C_{\lambda_0}\backslash\widetilde{\Sigma}_{\lambda_0}}\big[\frac1{|x-y|^{n-\alpha}}-\frac1{|x^{\lambda_0}-y|^{n-\alpha}}\big]{u^{p-1}(y)}dy.\label{UB-10}
\end{eqnarray}
If there exists some point $x_0\in\Sigma_{\lambda_0}$ such that $u(x_0)= u_{\lambda_0}(x_0)$, then  since $|x-y|>|x^{\lambda_0}-y|$ for $x\in{\Sigma}_{\lambda_0},y\in\Sigma^C_{{\lambda_0}}$, we deduce  from \eqref{UB-10} that
\[u(y)\equiv \infty,\quad \forall y\in\Sigma^C_{\lambda_0}\backslash\widetilde{\Sigma}_{\lambda_0}.\]
This contradicts to the  assumption that $u\in C^1(\overline{B_1})$ is a positive solution.

For some small $\delta_1>0$, we choose $\varepsilon \in (0, \epsilon_1)$  small enough  such that for any  $\lambda\in [\lambda_0,\lambda_0+\varepsilon)$, there holds
\[u(x)\ge u_{\lambda}(x), \forall x\in \Sigma_{\lambda_0-\varepsilon_1},\]
and
\[
|\frac1{|x-y|^{n-\alpha}}-\frac1{|x^{\lambda}-y|^{n-\alpha}}|\le\delta_1 ~~~~~~~~~ \quad\text{for}~x\in\Sigma_\lambda\backslash\Sigma_{\lambda_0-\varepsilon_1}.
\] 

Write
\begin{equation*}
\Sigma^u_\lambda=\{x\in \Sigma_\lambda |u_\lambda(x)>u(x)\}.
\end{equation*}
It follows from \eqref{UB-3} that for any $x\in\Sigma^u_\lambda$,
\begin{eqnarray*}
0> u(x)-u_\lambda(x)&\ge&\int_{\Sigma_\lambda}\big[\frac1{|x-y|^{n-\alpha}}-\frac1{|x^\lambda-y|^{n-\alpha}}\big]({u^{p-1}(y)}-{u^{p-1}_\lambda(y)})dy\\
&\ge&\int_{\Sigma_\lambda^u}\big[\frac1{|x-y|^{n-\alpha}}-\frac1{|x^\lambda-y|^{n-\alpha}}\big]({u^{p-1}(y)}-{u^{p-1}_\lambda(y)})dy\\
&\ge&-\delta_1\int_{\Sigma_\lambda^u}({u^{p-1}(y)}-{u^{p-1}_\lambda(y)})dy.
\end{eqnarray*}
 Since $u\in C^1(\overline{B_1})$,  there exists a positive constant $C_0$ such that $\frac 1{C_0}\le u\le C_0$.   It  follows from the above
\begin{eqnarray*}
 \int_{\Sigma_\lambda^u}(u_\lambda(x)-u(x))dx
&\le&\delta_1\int_{\Sigma_\lambda^u}\int_{\Sigma_\lambda^u}({u^{p-1}(y)}-{u^{p-1}_\lambda(y)})dydx\\
&\le& (1-p)\delta_1\int_{\Sigma_\lambda^u}\int_{\Sigma_\lambda^u}u^{p-2}(y)({u_\lambda(y)}-{u(y)})dydx\\
&\le&C\delta_1(\varepsilon+\varepsilon_1)^n\int_{\Sigma_\lambda^u}({u_\lambda(y)}-{u(y)})dy.
\end{eqnarray*}
It implies that
\[\|u_{\lambda}-u\|_{L^1(\Sigma^u_\lambda)}\equiv 0,\]
 for $\delta_1, \varepsilon, \varepsilon_1$ small enough, and hence $\Sigma^u_\lambda$ must have measure zero.

We thus have
\[u(x)-u_{\lambda}(x)\geq 0, \quad~ \text{for~ any} ~x\in \Sigma_\lambda, ~\forall\lambda\in [\lambda_0,\lambda_0+\varepsilon)\]
since $u$ is continuous. This contradicts to the definition of $\lambda_0$.
Hence, $\lambda_0=0$. We hereby  complete the proof of Theorem 3.4.
\hfill$\Box$
\medskip

\section{Existence result for critical case}

In this section,  we study the existence of positive solutions to the integral equation with critical exponent.

The non-existence of positive solution to \eqref{HB} with critical exponent for $\lambda \ge 0$ on a start-shaped domain follows from  Pohozaev identity \eqref{SB-5}. Next, we shall establish  the existence as well as the regularity results for weak  solutions to \eqref{HB}  with critical exponent for $\lambda<0$. To this end, we consider
\begin{equation*}
Q_\lambda(\Omega):=\inf_{f\in L_+^{q_\alpha}(\Omega)}\frac{\int_{\Omega} \int_{\Omega}f(x)  (|x-y|^{-(n-\alpha)}+\lambda|x-y|^{-(n-\alpha-1)})f(y)dy dx}{\|f\|^2_{L^{q_\alpha}(\Omega)}}.
\end{equation*}
 Notice that the corresponding Euler-Lagrange equation for extremal functions, up to a constant multiplier,  is  integral equation \eqref{HB} with $q=q_\alpha$.

First, we show
\begin{lemma}\label{Lm-inequ}
$Q_\lambda(\Omega)<N_\alpha$ for all $\lambda<0$. Further, $0<Q_\lambda(\Omega)<N_\alpha$ for any $\lambda\in (-\frac{1}{d(\Omega)}, 0)$.
\end{lemma}
\begin{proof} Let $x_*\in \Omega$.
For small positive $\epsilon$ and a fixed $R>0$ so that $B_R(x_*) \subset \Omega$, we define
 \begin{eqnarray*}
\tilde{f}_\epsilon(x)= \begin{cases}f_\epsilon(x) &\quad x\in B_{R}(x_*)\subset\Omega,\\
 0&\quad x\in\mathbb{R}^n\backslash B_{R}(x_*),
  \end{cases}
  \end{eqnarray*}
where $f_\epsilon$ is given by \eqref{fe}.
Obviously, $\tilde{f}_\epsilon\in L^{q_\alpha}(\mathbb{R}^n).$ Thus,  similar to the proof of Proposition \ref{prop2-1},  we have
\begin{eqnarray*}\label{BE-1}
& &\int_{\Omega} \int_{\Omega} \big(\frac{1}{|x-y|^{n-\alpha}}+\frac{\lambda}{|x-y|^{n-\alpha-1}}\big)\tilde{f}_\epsilon(x)\tilde{f}_\epsilon(y) dx dy\\
&&=\int_{\mathbb{R}^n} \int_{\mathbb{R}^n} \frac{1}{|x-y|^{n-\alpha}}{f_\epsilon}(x){f_\epsilon}(y) dx dy\\
& &-2\int_{\mathbb{R}^n} \int_{\mathbb{R}^n\backslash B_{R}(x_*) } \frac{ f_\epsilon(x) f_\epsilon(y)}{|x-y|^{n-\alpha}} dx dy
+\int_{\mathbb{R}^n\backslash B_{R}(x_*) } \int_{\mathbb{R}^n\backslash B_{R}(x_*) } \frac{ f_\epsilon(x) f_\epsilon(y)}{|x-y|^{n-\alpha}} dx dy\\
& &+\lambda \int_{B_{R}(x_*)} \int_{B_{R}(x_*)}\frac{f_\epsilon(x)f_\epsilon(y)}{|x-y|^{n-\alpha-1}} dx dy\\
&\le&N_\alpha\|f_\epsilon\|^2_{L^{q_\alpha}(\mathbb{R}^n)}-C_1(\frac\epsilon R)^n+\lambda J_1,
\end{eqnarray*}
where
\begin{eqnarray*}
J_1&:=&\int_{B_{R}(x_*)} \int_{B_{R}(x_*)}\frac{f_\epsilon(x)f_\epsilon(y)}{|x-y|^{n-\alpha-1}} dx dy\\
&=&\int_{B_{R}(x_*)} \int_{B_{R}(x_*)}{|x-y|^{-(n-\alpha-1)}}
\big(\frac{\epsilon}{\epsilon^2+|x-x_*|^2}\big)^\frac{n+\alpha}2\big(\frac{\epsilon}{\epsilon^2+|y-x_*|^2}\big)^\frac{n+\alpha}2 dx dy\\
&=&\epsilon^{-(n-\alpha-1)-(n+\alpha)} \int_{B_{R}(0)} \int_{B_{R}(0)}{\big|\frac{x-y}\epsilon\big|^{-(n-\alpha-1)}}
\big(1+\big|\frac{x}{\epsilon}\big|^2\big)^{-\frac{n+\alpha}2}\big(1+\big|\frac{y}{\epsilon}\big|^2\big)^{-\frac{n+\alpha}2 }dx dy\\
&=&\epsilon\int_{B_{\frac{R}\epsilon}(0)} \int_{B_{\frac{R}\epsilon}(0)}{|\xi-\eta|^{-(n-\alpha-1)}}
\big(1+|\xi|^2\big)^{-\frac{n+\alpha}2}\big(1+|\eta|^2\big)^{-\frac{n+\alpha}2 }d\xi d\eta\\
&\ge&C_0\epsilon.
\end{eqnarray*}
So, for $\lambda<0$ and small enough $\epsilon>0$, we have
\begin{eqnarray*}
& &\int_{\Omega} \int_{\Omega} \big(\frac{1}{|x-y|^{n-\alpha}}+\frac{\lambda}{|x-y|^{n-\alpha-1}}\big)\tilde{f}_\epsilon(x)\tilde{f}_\epsilon(y) dx dy\\
&\le&N_\alpha\|f_\epsilon\|^2_{L^{q_\alpha}(\mathbb{R}^n)}-C_1(\frac{\epsilon}R)^n+\lambda C_0\epsilon.
\end{eqnarray*}
This implies that $Q_\lambda(\Omega)<N_\alpha$ for all $\lambda<0$.

On the other hand, it is easy to see that $Q_\lambda(\Omega)>0$ for any $\lambda\in (-\frac{1}{d(\Omega)}, 0)$.
\end{proof}

The existence of solutions to equation \eqref{HB} will follow from
 the existence of a minimizer for energy $Q_\lambda(\Omega)$.

\begin{proposition}\label{sub-existence-2}
For any  $\lambda\in (-\frac{1}{d(\Omega)}, 0)$, infimum $Q_\lambda(\Omega) $ is achieved by a positive function $f_* \in L^{q_\alpha}(\Omega).$
\end{proposition}

For $q<q_\alpha,$  consider
$$Q_{\lambda, q}(\Omega)=\inf_{f\in L_+^{q}(\Omega)}\frac{ \int_{\Omega} \int_{\Omega} f(x)|x-y|^{-(n-\alpha)} f(y) dx dy+\lambda \int_{\Omega} \int_{\Omega} f(x)|x-y|^{-(n-\alpha-1)} f(y) dx dy}{\|f\|^2_{L^{q}(\Omega)}}.
$$

 By Lemma \ref{Bound-attained}, the infimum is attained by a positive function $f_q$, which satisfies
the subcritical equation
\begin{equation}\label{sub-equ2}
Q_{\lambda, q}(\Omega) f^{q-1}(x)=\int_\Omega\frac{f(y)}{|x-y|^{n-\alpha}}dy+ \lambda \int_\Omega\frac {f(y)}{|x-y|^{n-\alpha-1}}dy,\quad x\in\overline{\Omega},
\end{equation}
and $\|f_q\|_{L^{q}(\Omega)}=1.$ Further, we can show easily that $f_q \in C(\overline \Omega)$  and  $Q_{\lambda, q} \to Q_\lambda$ for $q \to (q_\alpha)^-.$

%
\begin{lemma}\label{lem3-3}For  $\lambda\in (-\frac{1}{d(\Omega)}, 0)$  and $q\in (0, q_\alpha)$,  let $f_q>0$ be a minimal energy solution to \eqref{sub-equ2} with $\|f_q\|_{L^{q}(\Omega)}=1.$   If  $0<Q_{\lambda, q}\le N_\alpha-\epsilon$ for some $\epsilon>0$, then there exists $C>0$ such that $\frac{1}{C}\le f_q(x)\le C$ uniformly for all  $x\in\overline{\Omega}$ and  $q\in (0, q_\alpha)$.
\end{lemma}

\begin{proof}
 It is  easy to see that  $\max\limits_{\overline \Omega} f_q(x):=f_q(x_q)\le C<\infty$ uniformly for all  $x\in\overline{\Omega}$ and  $q\in (0, q_1)$ provided $0<q_1<q_\alpha$.

We first prove by contradiction that  $\max\limits_{\overline \Omega} f_q(x)=f_q(x_q)\le C<\infty$ uniformly for all  $x\in\overline{\Omega}$ and  $q\in (0, q_\alpha)$. Suppose not. Then $f_q(x_q) \to +\infty$ for $q \to (q_\alpha)^-$. Let
$$\mu_q= f_q^{-\frac {2-q}\alpha}(x_q), \, \, \mbox{ and} \, \, \Omega_\mu=\frac{\Omega-x_q}{\mu_q}:=\{z \ | \ z=\frac {x-x_q}{\mu_q} \ \mbox{for} ~\, x \in \Omega\}.
$$
Define
\begin{equation}\label{blowup-1}
g_q(z)=\mu_q^{\frac{\alpha}{2-q}} f_q(\mu_q z+x_q), \,  \ \, \, \mbox{for} \, \, z \in \overline{\Omega}_\mu.
\end{equation}
Then $g_q$ satisfies
\begin{equation}\label{sub-equ3}
Q_{\lambda, q}(\Omega) g_q^{q-1}(z)=\int_{\Omega_\mu}\frac{g_q(y)}{|z-y|^{n-\alpha}}dy+ \lambda \mu_q \int_{\Omega_\mu}\frac {g_q(y)}{|z-y|^{n-\alpha-1}}dy,\quad z\in\overline{\Omega}_\mu,
\end{equation}
and $g_q(0)=1$, $g_q(z) \in (0, 1]$.

For convenience, denote $h_q(z):=g_q^{q-1}(z)$. Then \eqref{sub-equ3} is equivalent to
\begin{equation}\label{sub-equ4}
Q_{\lambda, q}(\Omega) h_q(z)=\int_{\Omega_\mu}\frac{h_q^{p-1}(y)}{|z-y|^{n-\alpha}}dy+ \lambda \mu_q \int_{\Omega_\mu}\frac {h_q^{p-1}(y)}{|z-y|^{n-\alpha-1}}dy,\quad z\in\overline{\Omega}_\mu,
\end{equation}
where $\frac{1}{p}+\frac{1}{q}=1, \ h_q(0)=1, \ h_q(z) \ge 1$.

Claim:  There exist $C_1, C_2>0$ such that, for all $z$ in a domain $\widehat \Omega$ covered by $\Omega_\mu$, when $q \to (q_\alpha)^-$
\begin{equation}\label{sub-equ7}
0<C_1(1+|z|^{\alpha-n})\le h_q(z) \le C_2(1+|z|^{\alpha-n}), \ \mbox{uniformly.}
\end{equation}

We relegate the proof of this claim to the end.

Once the claim is proved, we can prove that $h_q(z)$ is equicontinuous on any bounded domain $\widehat{\Omega}\subset \Omega_\mu$ when $q \to (q_\alpha)^-$. We write
\begin{eqnarray*}
&&Q_{\lambda, q}(\Omega)h_q(z)\\
&=&\int_{\Omega_\mu\setminus B(0,R)}\frac{h_q^{p-1}(y)}{|z-y|^{n-\alpha}}dy+\int_{\Omega_\mu\cap B(0,R)}\frac{h_q^{p-1}(y)}{|z-y|^{n-\alpha}}dy\\
&+&\lambda \mu_q \int_{\Omega_\mu\setminus B(0,R)}\frac {h_q^{p-1}(y)}{|z-y|^{n-\alpha-1}}dy+\lambda \mu_q \int_{\Omega_\mu\cap B(0,R)}\frac {h_q^{p-1}(y)}{|z-y|^{n-\alpha-1}}dy.
\end{eqnarray*}
Notice that
\begin{equation*}
\int_{\Omega_\mu\setminus B(0,R)}\frac{h_q^{p-1}(y)}{|z-y|^{n-\alpha}}(1+\lambda \mu_q|z-y|)dy \ge (1-|\lambda|d(\Omega))\int_{\Omega_\mu\setminus B(0,R)}\frac{h_q^{p-1}(y)}{|z-y|^{n-\alpha}}dy \ge0.
\end{equation*}
 Then for $\epsilon>0$ small enough,  we have, by \eqref{sub-equ7}, that
\begin{eqnarray}\nonumber
0&\le&\int_{\Omega_\mu\setminus B(0,R)}\frac{h_q^{p-1}(y)}{|z-y|^{n-\alpha}}(1+\lambda \mu_q|z-y|)dy\\\nonumber
&\le &C\int_{\Omega_\mu\setminus B(0,R)}\frac{h_q^{p-1}(y)}{|y|^{n-\alpha}}dy\\\nonumber
&\le&C\int_R^\infty r^{(\alpha-n)(p-1)+\alpha-1}dr\\\label{R-small-term}
&=&C R^{(\alpha-n)(p-1)+\alpha}<\epsilon
\end{eqnarray}
for any $z\in\widehat{\Omega}$, by taking $R>0$ large enough and  $q$ close to $q_\alpha$. Similarly, we have
\begin{equation}\label{R-small-term-1}
|\lambda \mu_q \int_{\Omega_\mu\cap B(0,R)}\frac {h_q^{p-1}(y)}{|z-y|^{n-\alpha-1}}dy|<\epsilon
\end{equation}
by taking $R>0$ large enough and $q$ close to $q_\alpha$. On the other hand, it is easy to see that $\int_{\Omega_\mu\cap B(0,R)}\frac{h_q^{p-1}(y)}{|z-y|^{n-\alpha}}dy\in C^1(\widehat{\Omega})$. Hence
for $z_1,z_2\in \widehat{\Omega}$,
\begin{eqnarray}\nonumber
&&|\int_{\Omega_\mu\cap B(0,R)}\frac{h_q^{p-1}(y)}{|z_1-y|^{n-\alpha}}dy-\int_{\Omega_\mu\cap B(0,R)}\frac{h_q^{p-1}(y)}{|z_2-y|^{n-\alpha}}dy|\\\nonumber
&\le& \int_{\Omega_\mu\cap B(0,R)}h_q^{p-1}(y)\frac{1}{|\xi-y|^{n-\alpha+1}}dy|z_1-z_2| \\\label{R-main-term}
&\le&  \int_{B(0,R)}\frac{1}{|\xi-y|^{n-\alpha+1}}dy|z_1-z_2|\le C R^{\alpha-1}|z_1-z_2|,
\end{eqnarray}
where $\xi=tz_1+(1-t)z_2$ for some $t\in(0,1)$. By \eqref{R-small-term}, \eqref{R-small-term-1} and \eqref{R-main-term} we conclude that $h_q(z)$ is equicontinuous on bounded domain $\widehat{\Omega}\in\mathbb{R}^n$ when $q \to (q_\alpha)^-$.

 As $q\to (q_\alpha)^-$, there are two cases:

\noindent{Case 1.} \  $\Omega_\mu \to \Bbb{R}_T^n:= \{(z_1, z_2, \cdots, z_n) \ | \ z_n >T\ge 0\},$ and $h_q(z) \to h(z)\in C({\Bbb{R}_T^n})$ uniformly  in any compact set in $\Bbb{R}_T^n$, where $h(z)$ satisfies
\begin{equation}\label{blowup-3}
Q_\lambda h(z)= \int_{\Bbb{R}_T^n} \frac {h^{p_\alpha-1}(y)}{|z-y|^{n-\alpha}} dy, \ \ \ \ h(0)=1.
\end{equation}
Also, direct computation yields
$$1=\int_\Omega f_q^q(y) dy= \mu_q^{(\frac{n+\alpha}{2-q})\cdot (\frac {2n}{n+\alpha}-q)}\cdot \int_{\Omega_\mu} g_q^qdz \le \int_{\Omega_\mu} g_q^qdz=\int_{\Omega_\mu} h_q^pdz.
$$
On the other hand, by \eqref{sub-equ7} we have $\int_{\Omega_\mu} h_q^pdz\le C$ uniformly.
 Again by \eqref{sub-equ7}, $\int_{\Bbb{R}_T^n} h^{p_\alpha} dz=\lim\limits_{q\to (q_\alpha)^-}\int_{\Omega_\mu} h_q^pdz \ge 1.$  Denote $g(x)=h^{p_\alpha-1}(x)$. Then $\int_{\Bbb{R}_T^n} h^{p_\alpha} dz=\int_{\Bbb{R}_T^n} g^{q_\alpha} dz$. By \eqref{blowup-3}, we have
\begin{eqnarray*}
N_\alpha-\epsilon\ge Q_\lambda&=& \frac{\int_{\Bbb{R}_T^n} \int_{\Bbb{R}^n} \frac {g(x) g(y)}{|x-y|^{n-\alpha}} dx dy}{\|g\|^{q_\alpha}_{L^{q_\alpha}(\Bbb{R}_T^n)}}
\\
&\ge& \frac{\int_{\Bbb{R}_T^n} \int_{\Bbb{R}_T^n} \frac {g(x) g(y)}{|x-y|^{n-\alpha}} dx dy}{\|g\|^{2}_{L^{q_\alpha}(\Bbb{R}_T^n)}}\\
&=&\frac{\int_{\Bbb{R}^n} \int_{\Bbb{R}^n} \frac {\tilde{g}(x) \tilde{g}(y)}{|x-y|^{n-\alpha}} dx dy}{\|\tilde{g}\|^{2}_{L^{q_\alpha}(\Bbb{R}^n)}}\ge N_\alpha.
\end{eqnarray*}
Contradiction!

\noindent {Case 2.} \ $\Omega_\mu \to \Bbb{R}^n,$ and $h_q(z) \to h(z)\in C({\Bbb{R}^n})$ uniformly  in any compact set in $\Bbb{R}^n$, where $h(z)$ satisfies
\begin{equation}\label{blowup-4}
Q_\lambda h(z)= \int_{\Bbb{R}^n} \frac {h^{p_\alpha-1}(y)}{|z-y|^{n-\alpha}} dy, \ \ \ \ h(0)=1.
\end{equation}
Similarly, $C\ge \int_{\Bbb{R}^n} h^{p_\alpha} dz\ge 1.$  Denote $g(x)=h^{p_\alpha-1}(x)$. Then $\int_{\Bbb{R}^n} h^{p_\alpha} dz=\int_{\Bbb{R}^n} g^{q_\alpha} dz$. By \eqref{blowup-4} we have
$$
N_\alpha-\epsilon\ge Q_\lambda= \frac{\int_{\Bbb{R}^n} \int_{\Bbb{R}^n} \frac {g(x) g(y)}{|x-y|^{n-\alpha}} dx dy}{\|g\|^{q_\alpha}_{L^{q_\alpha}(\Bbb{R}^n)}}\ge \frac{\int_{\Bbb{R}^n} \int_{\Bbb{R}^n} \frac {g(x) g(y)}{|x-y|^{n-\alpha}} dx dy}{\|g\|^{2}_{L^{q_\alpha}(\Bbb{R}^n)}}\ge N_\alpha,
$$
which again implies a contradiction.

Thus we conclude that there exists $C>0$ such that $f_q(y)\le C$ uniformly in $y\in\overline{\Omega}$ and  $q\in (0, q_\alpha)$.


On the other hand, if $\min\limits_{\overline \Omega} f_q(x):=f_q(\tilde{x}_q) \to 0$ as $q \to (q_\alpha)^-$, by using $f_q(y)\le C$ uniformly in $y\in\overline{\Omega}$ and  $q\in (0, q_\alpha)$ we have
$$\infty\leftarrow f_q^{q-1}(\tilde{x}_q)=\int_\Omega\frac{f_q(y)}{|\tilde{x}_q-y|^{n-\alpha}}dy+ \lambda \int_\Omega\frac {f_q(y)}{|\tilde{x}_q-y|^{n-\alpha-1}}dy\le C<\infty$$
as $q \to (q_\alpha)^-$, which gives a contradiction.

\smallskip

Now we are left to prove claim \eqref{sub-equ7}.

We first notice that
\begin{equation}\label{sub-equ5}
Q_{\lambda, q}(\Omega)=Q_{\lambda, q}(\Omega)h_q(0)=\int_{\Omega_\mu}\frac{h_q^{p-1}(y)}{|y|^{n-\alpha}}(1+\lambda \mu_q|y|)dy.
\end{equation}
Thus,
\begin{eqnarray}\nonumber
\int_{\Omega_\mu}h_q^{p-1}(y)|y|^{\alpha-n}dy&=& \frac{1}{(1-|\lambda| d(\Omega))}\int_{\Omega_\mu}h_q^{p-1}(y)|y|^{\alpha-n}(1-|\lambda| d(\Omega))dy\\\nonumber
&\le & \frac{1}{(1-|\lambda| d(\Omega))}\int_{\Omega_\mu}h_q^{p-1}(y)|y|^{\alpha-n}(1+\lambda \mu_q|y|)dy\\\label{sub-equ5-1}
&\le & C<\infty,
\end{eqnarray}
uniformly as $q \to (q_\alpha)^-$. Since $h_q \ge 1$ and $p<0$, we have
\begin{eqnarray}\label{sub-equ5-2}
\int_{\Omega_\mu}h_q^{p-1}(y)dy\le C<\infty
\end{eqnarray}
uniformly as $q \to (q_\alpha)^-$.

On the other hand, we also have
\begin{equation}\label{sub-equ6}
\int_{\Omega_\mu}h_q^{p-1}(y)dy\ge c_0>0, \ \ \mbox{as} \ \ q \to (q_\alpha)^-.
\end{equation}
Otherwise, if  there exists a sequence $q_n \to (q_\alpha)^-$ such that $\int_{\Omega_\mu}h_{q_n}^{p_n-1}(y)dy\to 0$ with $\frac{1}{p_n}+\frac{1}{q_n}=1$,
then for given $R_0>0$ and $\epsilon>0$ small we can take $R>>R_0$ large enough, such  that for $z\in \Omega_\mu\cap B(0, R_0 )$, as $q_n$ close to $(q_\alpha)^-$,
\begin{eqnarray*}
1&\le & h_{q_n}(z)\\
&=&\frac{1}{Q_{\lambda, q_n}(\Omega)}\big( \int_{\Omega_\mu\setminus B(0,R)}\frac{h_{q_n}^{p_n-1}(y)}{|z-y|^{n-\alpha}}dy+\int_{\Omega_\mu\cap B(0,R)}\frac{h_{q_n}^{p_n-1}(y)}{|z-y|^{n-\alpha}}dy\\
&&+\lambda \mu_{q_n} \int_{\Omega_\mu\setminus B(0,R)}\frac {h_{q_n}^{p_n-1}(y)}{|z-y|^{n-\alpha-1}}dy+\lambda \mu_{q_n} \int_{\Omega_\mu\cap B(0,R)}\frac {h_{q_n}^{p_n-1}(y)}{|z-y|^{n-\alpha-1}}dy\big)\\
&\le&\frac{1}{Q_{\lambda, q_n}(\Omega)}\big( \int_{\Omega_\mu\setminus B(0,R)}\frac{h_{q_n}^{p_n-1}(y)}{|z-y|^{n-\alpha}}dy+\int_{\Omega_\mu\cap B(0,R)}\frac{h_{q_n}^{p_n-1}(y)}{|z-y|^{n-\alpha}}dy\\
&&+\lambda \mu_{q_n} \int_{\Omega_\mu\setminus B(0,R)}\frac {h_{q_n}^{p_n-1}(y)}{|z-y|^{n-\alpha-1}}dy\big) \\
&\le& \frac{1}{Q_{\lambda, {q_n}}(\Omega)}\big( (1+\frac{R_0}{R})^{\alpha-n}\int_{\Omega_\mu\setminus B(0,R)}\frac{h_{q_n}^{p_n-1}(y)}{|y|^{n-\alpha}}dy+(R+R_0)^{\alpha-n}\int_{\Omega_\mu\cap B(0,R)}h_{q_n}^{p_n-1}(y)dy\\
&&+\lambda \mu_{q_n} (1-\frac{R_0}{R})^{\alpha+1-n}\int_{\Omega_\mu\setminus B(0,R)}\frac {h_{q_n}^{p_n-1}(y)}{|y|^{n-\alpha-1}}dy\big)\\
&=& \frac{1}{Q_{\lambda, {q_n}}(\Omega)}\big( (1+\frac{R_0}{R})^{\alpha-n}\int_{\Omega_\mu\setminus B(0,R)}\frac{h_{q_n}^{p_n-1}(y)}{|y|^{n-\alpha}}dy+(R+R_0)^{\alpha-n}\int_{\Omega_\mu\cap B(0,R)}h_{q_n}^{p_n-1}(y)dy\\
&&+\lambda \mu_{q_n} (1-\frac{R_0}{R})^{\alpha+1-n}\int_{\Omega_\mu}\frac {h_{q_n}^{p_n-1}(y)}{|y|^{n-\alpha-1}}dy-\lambda \mu_{q_n} (1-\frac{R_0}{R})^{\alpha+1-n}\int_{\Omega_\mu\cap B(0,R)}\frac {h_{q_n}^{p_n-1}(y)}{|y|^{n-\alpha-1}}dy\big)\\
&\le & \frac{1}{Q_{\lambda, {q_n}}(\Omega)}\big( (1+\frac{R_0}{R})^{\alpha-n}\int_{\Omega_\mu}\frac{h_{q_n}^{p_n-1}(y)}{|y|^{n-\alpha}}dy+\lambda \mu_{q_n} (1-\frac{R_0}{R})^{\alpha+1-n}\int_{\Omega_\mu}\frac {h_{q_n}^{p_n-1}(y)}{|y|^{n-\alpha-1}}dy\big)\\
&&+(R+R_0)^{\alpha-n}\frac{1}{Q_{\lambda, {q_n}}(\Omega)}\int_{\Omega_\mu}h_{q_n}^{p_n-1}(y)dy-\lambda \mu_{q_n} (R-R_0)^{\alpha+1-n}\frac{1}{Q_{\lambda, {q_n}}(\Omega)}\int_{\Omega_\mu}h_{q_n}^{p_n-1}(y)dy\\
&\le & 1+\epsilon.
\end{eqnarray*}
That is, $h_{q_n}(z)\to 1, \ z\in \Omega_\mu\cap B(0,R_0)$ uniformly as $q_n \to (q_\alpha)^-$. Then for $R_0>0$ large, since
$\int_{\Omega_\mu\cap B(0, R_0) }h_{q_n}^{p_n-1}(y)dy\le \int_{\Omega_\mu}h_{q_n}^{p_n-1}(y)dy$  and $\Omega_\mu$ goes to either $\Bbb{R}_T^n:= \{(z_1, z_2, \cdots, z_n) \ | \ z_n >T\ge 0\}$ or $\Bbb{R}^n,$  we obtain a contradiction to \eqref{sub-equ5-2}.

By \eqref{sub-equ4}, we have
\begin{eqnarray}\label{sub-equ7-1}
&&\lim_{|z|\to\infty}\frac{1}{|z|^{\alpha-n}}\int_{\Omega_\mu}\frac{h_q^{p-1}(y)}{|z-y|^{n-\alpha}}dy\nonumber\\
&\ge&\lim_{|z|\to\infty} Q_{\lambda, q}(\Omega) \frac{h_q(z)}{|z|^{\alpha-n}}
=\lim_{|z|\to\infty}\frac{\int_{\Omega_\mu}\frac{h_q^{p-1}(y)}{|z-y|^{n-\alpha}}(1+ \lambda \mu_q |z-y|)dy}{|z|^{\alpha-n}}\nonumber\\
&\ge&(1-|\lambda|d(\Omega))\lim_{|z|\to\infty}\frac{\int_{\Omega_\mu}\frac{h_q^{p-1}(y)}{|z-y|^{n-\alpha}}dy}{|z|^{\alpha-n}}.
\end{eqnarray}
Since $\frac{1}{|z|^{\alpha-n}}\frac{h_q^{p-1}(y)}{|z-y|^{n-\alpha}}\le 2^{\alpha-n}h_q^{p-1}(y)(1+|y|^{\alpha-n})$ as $|z|\to \infty$,  and $\int_{\Omega_\mu}h_q^{p-1}(y)(1+|y|^{\alpha-n})dy\le C$
by \eqref{sub-equ5-1} and \eqref{sub-equ5-2}, we then have
\begin{equation}\label{sub-equ7-2}
\lim_{|z|\to\infty}\frac{1}{|z|^{\alpha-n}}\int_{\Omega_\mu}\frac{h_q^{p-1}(y)}{|z-y|^{n-\alpha}}dy =\int_{\Omega_\mu}h_q^{p-1}(y)dy.
\end{equation}
Hence by \eqref{sub-equ7-1}, \eqref{sub-equ7-2}, \eqref{sub-equ5-2} and \eqref{sub-equ6}  we obtain the claim \eqref{sub-equ7}.  Hereby we complete the proof of Lemma \ref{lem3-3}.
\end{proof}


\noindent{\bf Proof of Proposition \ref{sub-existence-2}}. \   Let $f_q>0$ being  solutions to \eqref{sub-equ2} for $q\in (0, q_\alpha)$, which are also the minimal energy functions to energy $Q_{\lambda, q}$. Then by Lemma \ref{lem3-3}, we know that  $\{f_q\}$ are uniformly bounded above and bounded below by a positive constant. Thus they are  equicontinuous due to equation  \eqref{sub-equ2}.  It follows that $f_q \to f_*$ as $q \to (q_\alpha)^-$ in $C(\overline \Omega),$ and $f_*$ is the energy minimizer for $Q_\lambda$.
 \hfill$\Box$

\medskip

\noindent{\bf Completion of the Proof of Theorem \ref{main}}.  Lemma \ref{Lm-inequ} and Proposition \ref{sub-existence-2} implies the existence of a positive solution  $f \in L^{p_\alpha}(\Omega)\cap C(\overline\Omega)$ to the equation \eqref{HB}  for $q=\frac{2n}{n+\alpha}$,  $\lambda\in (-\frac{1}{d(\Omega)}, 0)$. It is also easy to see that $f \in C^1(\overline\Omega)$. \hfill$\Box$

\medskip

 \vskip 1cm
\noindent {\bf Acknowledgements}\\
\noindent
We  dedicate this paper to Professor Ha\"im Brezis to celebrate his seventy five birthday. We thank him for his great influence on us in the study of elliptic equations through his lectures and numerous papers, in particular, the paper with Louis Nirenberg \cite{BN1983}.
The project is partially  supported by  the
National Natural Science Foundation of China (Grant No. 11571268) and the Fundamental Research Funds for the Central Universities (Grant No. GK201802015) and Simmons Collaboration (Grant No. 280487).


\end{document}